\documentclass[12pt,pdf,colorBG,reqno]{amsart}
\usepackage{amssymb,amsmath,amscd}
\usepackage{graphicx}
\usepackage{amsfonts}
\usepackage[all]{xy}

\input epsf
\usepackage[all]{xy}
\newcommand{\commentout}[1]{}

\newcommand{\Ind}{\mathrm{Ind}}

\newcommand{\Vol}{\mathrm{Vol}_\R}
\newcommand{\Gr}{\mathrm{Gr}}
\newcommand{\w}{\widetilde}

\renewcommand{\det}{\mathrm{det}}
\newcommand{\diag}{{\rm diag}}

\theoremstyle{plain}
\numberwithin{equation}{section}
\newtheorem{theorem}{Theorem}[section]
\newtheorem{corollary}[theorem]{Corollary}

\newtheorem{lemma}[theorem]{Lemma}

\theoremstyle{definition}

\newtheorem{remark}[theorem]{Remark}

\def\text#1{\;\;\;\;{\rm \hbox{#1}}\;\;\;\;}
\def\qquad{\quad\quad}

\def\msy#1{{\mathbb #1}}
\def\C{{\msy C}}
\def\N{{\msy N}}
\def\H{{\msy H}}

\def\K{{\msy K}}

\def\R{\mathbb{R}}
\def\ga{\alpha}
\def\gb{\beta}

\def\e{\epsilon}

\def\gD{\Delta}

\def\fa{{\mathfrak a}}
\def\fb{{\mathfrak b}}

\def\fg{{\mathfrak g}}
\def\fh{{\mathfrak h}}

\def\fk{{\mathfrak k}}
\def\fl{{\mathfrak l}}
\def\fm{{\mathfrak m}}
\def\fn{{\mathfrak n}}
\def\fp{{\mathfrak p}}
\def\fq{{\mathfrak q}}
\def\fs{{\mathfrak s}}

\def\fz{{\mathfrak z}}

\def\to{\rightarrow}
\def\Re{{\rm Re}\,}
\def\Im{{\rm Im}\,}
\def\LB{\Lambda^+(\cB )}

\def\Ad{{\rm Ad}}

\def\GL{\mathrm{GL}}
\newcommand{\SL}{\mathrm{SL}}
\def\Sp{\mathrm{Sp}}
\def\rM{\mathrm{M}}
\def\rG{\mathrm{G}}
\def\rI{\mathrm{I}}
\newcommand{\rS}{\mathrm{S}}
\newcommand{\rU}{\mathrm{U}}
\newcommand{\rO}{\mathrm{O}}

\def\Id{{\rm I}}

\def\ad{{\rm ad}}

\def\pr{{\rm pr}}

\newcommand{\id}{\mathrm{id}}

\def\cA{{\mathcal A}}
\def\cB{{\mathcal B}}

\def\cF{{\mathcal F}}
\def\cH{{\mathcal H}}

\def\cY{{\mathcal Y}}

\newcommand{\CL}{\mathcal{C}^\lambda}
\newcommand{\wKL}{\widehat{K}_L}
\newcommand{\Cos}{\mathrm{Cos}}
\newcommand{\CosL}{\mathrm{Cos}^{\lambda}}

\newcommand{\SinL}{\mathrm{Sin}^{\lambda}}
\newcommand{\SLa}{\mathcal{S}^\lambda}
\newcommand{\Res}{\mathrm{Res}}
\newcommand{\Kf}{\langle\, \cdot \, ,\, \cdot \, \rangle}

\newcommand{\oN}{\overline{N}}
\newcommand{\oP}{\overline{P}}

\def\rmO{{\rm O}}
\def\rmS{{\rm S}}
\def\rmU{{\rm U}}
\def\Tr{\mathrm{Tr}}

\def\SU{\rmS\rmU}
\def\SO{\rmS\rmO}

\def\sideremark#1{\ifvmode\leavevmode\fi\vadjust{\vbox to0pt{\vss
 \hbox to 0pt{\hskip\hsize\hskip1em%
 \vbox{\hsize2cm\tiny\raggedright\pretolerance10000
 \noindent #1\hfill}\hss}\vbox to8pt{\vfil}\vss}}}

\begin{document}

\title[Intertwining Operators and the cos${}^\lambda$-Transform]{The $\CosL$ and $\SinL$ Transforms\\ as Intertwining Operators between \\ generalized principal series Representations\\ of $\SL (n+1,\K)$}\author{Gestur {\'O}lafsson}
\address{Department of Mathematics, Louisiana State University, Baton Rouge, LA 70803}
\email{olafsson@math.lsu.edu}
\author{Angela Pasquale}
\address{Laboratoire de Math\'ematiques et Applications de Metz (UMR CNRS 7122),
Universit\'e Paul Verlaine Metz, F-57045 Metz, France.}
\email{pasquale@math.univ-metz.fr}
\thanks{The research by {\'O}lafsson was supported by NSF grant DMS-0801010 and the 
University Paul Verlaine of Metz}

\begin{abstract}
In this article we connect topics from convex and integral geometry with well known topics in representation theory of semisimple Lie groups by showing that the $\CosL$ and $\SinL$-transforms on the Grassmann manifolds $\Gr_p(\K)=\SU (n+1,\K)/\rS (\rU (p,\K)\times \rU (n+1-p,\K))$ are standard intertwining operators between certain generalized principal series representations induced from a maximal parabolic subgroup $P_p$ of $\SL (n+1,\K)$. The index ${}_p$ indicates the dependence of the parabolic on $p$. The general results of Knapp and Stein and Vogan and Wallach then show that both transforms have meromorphic extension to $\C$ and are invertible for generic $\lambda\in\C$. Furthermore, known methods from representation theory combined with a Selberg type integral allow us to determine the $K$-spectrum of those operators.
\end{abstract}
\maketitle
\section*{Introduction}
\setcounter{section}{0}
\noindent
The $\CosL$ and $\SinL$ transforms have been widely studied during the last few years. One of the reasons is their connection to convex geometry and to some classical integral transforms, like the Funk and Radon transform on the sphere and their generalizations to Grassmann manifolds. We will not discuss the history of these transforms but refer to  \cite{Rubin1998a,Rubin1998b,Rubin2008,Rubin2010} and references therein for further information.

For the sphere $\rS^n$, the $\CosL$ transform $\CL$ is defined on $L^2(\rS^n)$ by
\begin{equation}\label{eq-CosLIntro}
\CL  (f)(\omega )=\int_{\rS^n} |(x,\omega )|^{\lambda - \rho}\, f(x)\, d\sigma (x)
\end{equation}
where $d\sigma$ is the normalized rotational invariant measure on $\rS^n$, $(\cdot,\cdot)$ stands for the usual inner product on $\R^n$ and $\rho =(n+1)/2$. The name $\CosL$ transform comes from the fact that for non-zero $x$ and $y$ the quantity $(x,y)/\|x\|\|y\|$ is the cosine of the angle between $x$ and $y$. It is clear that $\CL$ is in fact an integral transform on $L^2 (\cB)$, where $\cB:=\Gr_1(\R)$ is the Grassmann manifold of lines in $\R^{n+1}$. 

Denote by $\Gr_n(\R)$ the Grassmann manifold of $n$-dimensional subspaces of $\R^{n+1}$.
The $\SinL$-transform $\SLa$ maps functions on $\cB$ to functions on $\Gr_n(\R)$ and is given by
\[\SLa  (f)(\omega) =\CL  (f)(\omega^\perp )\, .\]
Note that $\Gr_n(\R)\simeq \cB$ so that one can also view $\SLa $ as an integral transform on $L^2(\cB)$.

As we will see later, the factor $\rho$ in the exponent of (\ref{eq-CosLIntro}) is chosen so that $\CL$ and $\SLa$ coincide with will known integral transforms from representation theory, i.e., standard intertwining operators between certain principal series representations of $\SL (n+1,\R)$.

We note that $\CL (f)$ exists for all $f\in L^2(\rS^n)$ as long as $\lambda \in \C_\rho:=\{z\in\C\mid \Re (z)\ge \rho\}$ and $\lambda \mapsto 
\CL (f)$ is holomorphic on $\C_\rho^o$, the interior of $\C_\rho$. The transform $\CL : C^\infty (\rS^n)\to C^\infty (\rS^n)$ extends meromorphically to $\C$. One way to determine the singularities of this extension is to determine the $\SO (n+1)$-spectrum of $\CL$. The group $K=\SO (n+1)$ acts naturally on $L^2(\rS^n)$ and, as a representation of $K$,  the space $L^2 (\rS^n)$ decomposes into a direct sum
\[L^2(\rS^n)\simeq_K \bigoplus_{m\in \N_0} \cY^m\]
where
\[\cY^m=\{p|_{\rS^n}\mid p \text{a homogeneous polynomial of degree} m \text{and} \Delta p=0\}\, .\]
Here $\Delta=\partial_1^2+\cdots +\partial_{n+1}^2$ denotes the Laplace operator on $\R^{n+1}$. 
Likewise, by considering $\CosL$ as a transform on $L^2(\cB)$, we have the decomposition
\[L^2(\cB)\simeq_K \bigoplus_{m\in 2\N_0} \cY^m\, .\]
As these are multiplicity-one decompositions and $\CL$ commutes with rotations, it follows from Schur's Lemma that for each $m\in 2\N_0$ there exists a holomorphic  function $\eta_m : \C_\rho \to \C$ such that
\[\CL |_{\cY^m}=\eta_m(\lambda )\id_{\cY^m}\, .\]
Information about $\eta_m $ gives more detailed information about the meromorphic extension of $\CL$ as well as the image and the kernel of the $\CosL$-transform, as clearly both of those spaces are $K$-invariant. 

The functions $\eta_m(\lambda)$ can be for instance determined using Funk-Hecke type formula \cite{Rubin1998a,Rubin1998b}, and one gets for $m\in 2\N_0$
\[\eta_m(\lambda ) = (-1)^{m/2}\frac{\Gamma (\rho)}{\Gamma (1/2)}\, \frac{\Gamma \left (\frac{1}{2}(\lambda -\rho +1)\right)\Gamma \left(\frac{1}{2}(-\lambda +\rho + m)\right)}{\Gamma \left(\frac{1}{2}(\lambda +\rho +m)\right)\Gamma \left(\frac{1}{2}(-\lambda +\rho)\right)}\, .\]
Note that
\[\frac{\Gamma \left(\frac{1}{2}(-\lambda +\rho + m)\right)}{\Gamma \left(\frac{1}{2}(-\lambda +\rho)\right)}=\left\{\begin{matrix} 1 &\text{if} m=0\\
2^{-m/2}(-\lambda +\rho)(-\lambda+\rho+2) \cdots (-\lambda +\rho +m-2) &\text{if}m\in 2\N\end{matrix}\right.\]
is a polynomial of degree $m/2$. 

The first pole is at $\lambda = (n-1)/2$ and the Funk transform
\[\cF (f)(\omega )=\int_{x\cdot \omega =0} f(x)\]
is then given, up to a constant that can be determined, by  $\cF(f)(\omega )=c \Res_{\lambda = (n-1)/2} \CL  (f)(\omega)$.
A similar discussion relates the Radon transform with the singularities of the $\SinL$-transform.

All of these integral transforms have been generalized to the Grassmann manifolds over $\R$, $\C$ and $\H$. We refer to \cite{AB2004,OR06,Rubin2008,Rubin2009,Zhang09} for more details and discussions.

The integral transform in (\ref{eq-CosLIntro}) appeared independently in representation theory of $\SL (n+1,\K)$, where $\K=\R,\C$ or $\H$, around the same time as the $\CosL$-transform. The real case was studied in \cite{DM1999}, the complex case in \cite{DZ1997}, and finally the quaternionic case in \cite{P99}. In these articles it was shown that the integral (\ref{eq-CosLIntro}) defines an intertwining operator between generalized principal series representations induced from a maximal parabolic subgroup in $\SL (n+1,\K)$. The $K$-spectrum was determined. Among the applications was the composition series, some embeddings of the complementary series and the study of the so-called canonical representations on rank-one Riemannian symmetric spaces of the noncompact type, \cite{DHjfa,DH,DP}. The connections to convex geometry and to the Funk and Radon transforms was neither discussed nor mentioned.

Let $\K=\R,\C$ or $\H$ and let $\Gr_p(\K)=\SU (n+1,\K)/\rS (\rU (p,\K)\times \rU (q,\K))$, with  $q=n+1-p$, be the Grasmmann manifold of $p$-dimensional subspaces of $\K^{n+1}$.
Our main aim in this article is to show that the $\CosL$ and $\SinL$-transforms for $\Gr_p(\K)$ are always canonical intertwining operators between generalized  principal series representations induced from maximal parabolic subgroups of $\SL (n+1,\K)$. We then use the spectrum generating method developed in \cite{BOO} to determine the $K$-spectrum of the $\CosL$ and $\SinL$-transform. 
 
The article is organized as follows. In Section \ref{Section1} we recall needed notations and fact on simple Lie groups and Lie algebras. In Section \ref{section2} we introduce the generalized spherical principal series representations. These are representations induced from a character $\chi_\lambda$ of a parabolic subgroup $P=MAN$. The elements of the representation space can be viewed as $L^2$ sections of the $G$-homogeneous line bundle over $\cB=G/P$ corresponding to $\chi_\lambda$. We recall the results by Knapp and Stein and Vogan and Wallach on the meromorphic continuation of standard intertwining operators $J(\lambda)$ between spherical generalized principal series representations, \cite{KnappStein1,KnappStein2,VoganWallach}. 
The material of this section is presented in the general setting of semisimple Lie groups rather than for the special case considered later in this article. Indeed, the analytic properties of the intertwining operators which we discuss here possess a rather unified structure for the different semisimple Lie groups. On the other hand, we will associate a geometric meaning to these operators only for a specific class of semisimple Lie groups. It is therefore a natural question whether one can provide a geometric meaning to the standard intertwining operators for other groups or, more generally, to intertwining operators between generalized principal series representations induced from other parabolic subgroups. We expect that the answer will be positive and that higher rank $\CosL$-transforms are in fact given by such intertwining operators.

In Section \ref{section3} we recall the method from \cite{BOO} on how to find the eigenvalues of the intertwining operators on each of the $K$-types in case $P$ is a maximal parabolic subgroup and the $K$-types have multiplicity one. As the
situation considered in this article is less general then the one in \cite{BOO}, the results and their proofs are simpler. The principal result is Theorem \ref{th-3.6}, which provides us with an inductive way to compute the spectrum.

We specialize to the case of the Grassmann manifolds $\Gr_p(\K)$ in Section \ref{se-Grassman}. We introduce the $\CosL$-transform $\CL$ and 
show that $\CL =J(\lambda)$. This main result has several applications. Using the well-known fact that the generalized principal series are irreducible for almost all $\lambda\in \C$, we obtain that there exists a meromorphic function $c_P(\lambda )=J(\lambda )1=\CL  (1)$ such that $\mathcal{C}^{-\lambda}\circ \CL = \CL\circ \mathcal{C}^{-\lambda }=c_P(-\lambda ) c_P (\lambda )\id$ for almost all $\lambda$. 
Moreover, together with the method from \cite{BOO} outlined in Section \ref{section3}, the equality  $\CL =J(\lambda)$ allows us to calculate the $K$-spectrum of $\CL$. This computation is done in Section \ref{sec:Kspectrum}. The main result is Theorem \ref{th:etaMu}.  Finally, a similar discussion about the $\SinL$-transform is then carried out in Theorem \ref{th:KspectrumSL}.

\section{Semisimple Lie Groups and Parabolic Subgroups}\label{Section1}
\noindent
In this section we introduce the notation and basic facts on semisimple Lie groups, Lie algebras and parabolic subgroups.   The material can be found in any standard textbook on semisimple
Lie group, e.g. \cite{Knapp}.

{}From now on $G$ will stand for a noncompact connected semisimple Lie group with finite
center.  The Lie algebra of $G$ is denoted by $\fg$. In general Lie groups will be denoted by upper case Latin letters and their Lie algebras by the corresponding German letter.

An involution $\theta :G \to G$ is a Cartan involution if $K=G^\theta=\{a\in G\mid \theta (a)=a\}$ is a maximal compact subgroup of $G$. In that case $K$ is connected. From now on $\theta$ will denote a fixed Cartan involution on $G$ and $K$ the corresponding maximal compact subgroup. Our standard example will be $G=\SL (n+1,\K)$, $\K=\R$, $\C$, or $\H$, and $\theta (x)=(\overline{x}^t)^{-1}$ where $\bar{\hbox to 0.5em{}}$ denotes the canonical conjugation in $\K$ which fixes $\R$. In this case $K=\{x\in G\mid \overline{x}^t=x\}$. The general structure theory discussed in this section is specialized to this case in Section \ref{se-Grassman}.

The derived involution
$\dot{\theta} :\fg\to \fg$ and its complex linear extension to
$\fg_\C=\fg\otimes_\R \C$ will for simplicity also be denoted by $\theta$. Then $\fg=\fk\oplus\fs$ where
$\fk=\fg^\theta$ is the Lie algebra of $K$ and $\fs=\{X\in\fg\mid \theta (X)=-X\}$. We have $[\fs,\fs]\subseteq \fk$,  $[\fk,\fs]\subseteq \fs$, and $\Ad (K)\fs =\fs$. We denote the representation of $K$ on $\fs_\C$ by $\Ad_\fs$.

Denote by $\Kf$ the Killing form on $\fg$, given by $\langle X,Y\rangle=\Tr (\ad (X)\ad (Y))$. Then
\[(X,Y)\mapsto \langle X,Y\rangle_\theta:=\langle X, \theta (Y)\rangle\]
is an inner product on $\fg$. If $X\in \fg$
then $\ad (X)^*=-\ad (\theta (X))$. In particular, if $X\in\fs$ then
$\ad  (X )$ is a symmetric operator and hence diagonalizable over the reals. Let
$\fa\subset\fs$ be abelian. For $\alpha \in \fa^*$ let
\[\fg_\alpha =\{X\in \fg\mid (\forall H\in\fa)\, [H,X]=\alpha (H) X\}\]
and let
\[\fm=\{X\in \fg \mid (\forall H\in\fa)\, [H,X]=0\text{ and } X\perp \fa\}\, .\]
Then $\fm\oplus \fa=\fg_0$. We will always assume that the 
\[\fa =\fz (\fm)\cap \fs=\{X\in \fm\oplus \fa\mid [X,\fm\oplus \fa]=\{0\}\}\, .\]

Let $\Delta = \Delta (\fg,\fa):=\{\alpha \in \fa^*\setminus \{0\}\mid \fg_\alpha\not=\{0\}\}$.
The elements in $\Delta$ are called roots. We have
\[\fg=\fm \oplus \fa\oplus \bigoplus_{\alpha \in\Delta}\fg_\alpha\, .\]
As $[\fg_\alpha,\fg_\beta]\subseteq \fg_{\ga +\gb}$ for $\alpha,\beta\in\Delta\cup \{0\}$
it follows that $[\fm\oplus \fa,\fg_\ga]\subseteq \fg_\ga$
for all $\ga\in\gD$.

Let $\fa_r:=\{H\in\fa\mid (\forall \alpha \in \gD)\, \alpha (H)\not=0\}$. For a fixed $H\in\fa_r$
let $\gD^+:=\{\alpha \in\gD\mid \ga (H)>0\}$. Then $\gD^+$ is system of positive roots, i.e., $\Delta = \Delta^+\cup -\Delta^+$, $\Delta^+\cap -\Delta^+ =\emptyset$, and $(\Delta^+ + \Delta^+)\cap \Delta \subset \Delta^+$. In particular, 
\[\fn:=\bigoplus_{\ga\in\gD^+}\fg_\alpha\]
is a nilpotent subalgebra of $\fg$ normalized by the Lie algebra
$\fp:=\fm\oplus \fa\oplus \fn$.
In fact
\[\fp=\fn_{\fg}(\fn)=\{X\in \fg\mid [X,\fn]\subseteq \fn\}\, .\]
$\fp$ is a parabolic subalgebra of $\fg$. It is maximal, if $\dim \fa=1$.

Let $P=N_G (\fp)=\{g\in G\mid \Ad (g)\fp \subset \fp\}$. Then $P$ is a closed subgroup of $G$ with
Lie algebra $\fp$. Let $A=\exp \fa$ respectively $N= \exp \fn$ denote the analytic
subgroup of $G$ with Lie algebra $\fa$ respectively $\fn$. Then $A$ and $N$ are closed, $A$ is
abelian and $N$ is nilpotent. Let $M_o$ denote the analytic subgroup of
$G$ with Lie algebra $\fm$ and $M:=Z_K(A)M_o$. Then $M$ is a closed subgroup of $G$ with finitely many
connected components, $Z_G (A)=MA$ and  the product map $M\times A\times N\to P$,
$(m,a,n)\mapsto man$, is an analytic diffeomorphism. Let $L=K\cap M$ and $\cB := K/L$. Then
\begin{equation}\label{G=KP}
G= KP\quad \text{and}\quad K\cap P=L\, .
\end{equation}
Thus $\cB =G/P$ and $G$ acts on $\cB$. To describe this action, write
$g=k_P(g)m_P(g)a_P(g)n_P(g)$ with $(k_P(g),m_P(g),a_P(g),n_P(g))\in K\times M\times A\times N$. If the role of $P$ is clear, then we leave out the index ${}_P$. The map $G\ni g\mapsto (a(g),n(g))\in A\times N$ is analytic and $g\mapsto a(g)$ is right $MN$-invariant. We can
therefore view $a(\, \cdot \, )$ as a map $G/MN\to A$.
The elements $k(g)$ and $m(g)$ are not uniquely defined.
However, the maps
\[g\mapsto k(g)L\in \cB \text{ and }
g\mapsto L m(g)\in L\backslash M\]
are well defined and analytic. 
Let $b_o=eL$. Then we can describe the $G$ action on $\cB$ by $g\cdot (k\cdot b_o) =k(gk)\cdot b_o$.
If $b=k\cdot b_o\in\cB$ and $g\in G$, then we set $a(g b):= a(gk)$.

Let $\oN:= \theta (N)$. The Lie algebra of $\oN$ is $\bar \fn=\theta (\fn)=\bigoplus_{\ga\in -\gD^+}
\fg_\alpha$. Furthermore, $\fg=\bar\fn\oplus \fm\oplus \fa\oplus \fn=\bar\fn \oplus \fp$, and
the map $\oN\times M\times A\times N\to G$, $(\bar n,m, a,n)\mapsto \bar n man$, is
an analytic diffeomorphism onto an open and dense subset $\oN P$ of $G$ of full measure.

Let $\oP = MA\oN=\theta (P)$. Then as a $K$-manifold $\cB = G/\oP$ but the $G$-action in general is different. Denote this action for the moment by $\,\bar{\cdot}\,$. 

\begin{lemma}\label{le-BarAction} 
Let $g\in G$ and  $b\in \cB$. Then $a_{\oP} (gb) = a_P (\theta (g)b)^{-1}$, $k_{\oP} (gb)= k_{P} (\theta ( g)b)$ and $g\,\bar{\cdot}\, b=\theta(g)\cdot b$.
\end{lemma}
\begin{proof} Let $k\in K$ be such that $b=k\cdot b_o$. Then $g\,\bar{\cdot}\, b=k_{\oP}(gk)L$. We have
\begin{eqnarray*} gk &=&\theta (\theta (g)k)\\
&=&\theta (k_P (\theta (g)k)m_P(\theta (g)k)a_P(\theta (g)k)n_P(\theta (g)k)) \\
 &=&k_P (\theta (g)k)\theta( m_P(\theta (g)k))a_P(\theta (g)k)^{-1}\theta (n_P(\theta (g)k))\, .
 \end{eqnarray*}
 The Lemma now follows as $M$ is $\theta$-stable and $\theta (n_P(\theta (g)k))\in \oN$.
\end{proof}

We note for later use that the the above argument shows that
\begin{equation}\label{eq-APbar}
a_{\oP}(gk)=a_P(\theta (g)k)^{-1}\, .
\end{equation}

We normalize the invariant measure on $\cB$ and compact groups so that the total measure is one. If
$f\in C(\cB)$ then
\[\int_\cB f(b)\, db=\int_K f(k\cdot b_o)\, dk\, .\]

For $\lambda\in\fa_\C^*$ define a homomorphism $\chi_\lambda : P\to \C$ by
\[\chi_\lambda (m\exp H n)=e^{\lambda (H)}\, ,\quad m\in M,\, H\in \fa\, , \text{and}\, n\in N\, .\]
We will also use the shorthand notation $p^\lambda $ for $\chi_\lambda (p)$. For
$\alpha \in \gD$ let $m_\alpha :=\dim \fg_\alpha$ and let
$\rho_P : =\frac{1}{2}\sum_{\alpha\in\gD^+}m_\alpha \alpha
\in\fa^*$. As $P$ is fixed most of the time we simply write
$\rho$ instead of $\rho_P$. We have $p^{2\rho}=|\det \Ad (p)|_{\fn}|$.

Normalize the Haar measure $d\bar n$ on $\oN$ such that
\[\int_{\oN} a(\bar n)^{-2 \rho } \, d\bar n = 1\]
and let $dn=\theta (d\bar n)$. We normalize the invariant measure on $A$ respectively $\fa$ such that the inversion formula for the abelian Fourier transform holds without additional constants, i.e., if $f\in C^\infty_c(A)$ and $\widehat{f}(\lambda )=\int_{A} f(a)a^{-\lambda}\, da$, then $f(a) =\int_{\fa^*} \widehat{f}(i\lambda )a^{i\lambda }\, d\lambda$. 

\begin{lemma} Let the notation be as above.
\begin{enumerate}
\item We can normalize the Haar measure on $G$ and $M$ such that for all $f\in L^1(G)$
\begin{eqnarray*}
\int_G f(x)\, dx & =&\int_K\int_M\int_A \int_N f(kman)a^{2\rho }\, dkdmdadn\\
&=&\int_{\oN}\int_M \int_A \int_N f(\bar n man)a^{2\rho}\, d\bar ndmdadn\, .
\end{eqnarray*}
\item The set $\oN \cdot b_o\subset \cB $ is open, dense and of full measure. If $f\in L^1(\cB)$ then
\begin{equation}\label{intKNbar}
\int_\cB f(b)\, db=\int_{\oN} f(\bar n\cdot b_o)\, a (\bar n)^{-2\rho }\, d\bar n\, .
\end{equation}
\item Let $f\in L^1(\cB)$ and $x\in G$. Then
\begin{equation}\label{eq-RadonNikonB}
\int_{\cB}  f(x\cdot b)\, a(xb)^{-2\rho} \, db=\int_\cB f(b)\, db\, .
\end{equation}
\end{enumerate}
\end{lemma}
\begin{proof} See \cite{Knapp}, pp. 137--140 and p. 170, or \cite{O87}, \S 7.\end{proof}

\section{Generalized Principal Series Representations}\label{section2}
\noindent
We now  introduce the spherical generalized principal series
representations and their intertwining operators. We refer to
\cite{KnappStein1,KnappStein2,VoganWallach} for proofs. Later we will show that those intertwining operators coincide with the $\CosL$-transform.

For $\lambda\in\fa_\C^*$ define a continuous representation of $G$  on $L^2(\cB)$ by 
\begin{equation}\label{piLambda}
[\pi_\lambda (x)f] (b):= a(x^{-1}b )^{-\lambda-\rho}f(x^{-1}\cdot b)\, .
\end{equation}
Then $\pi_\lambda =\pi_{P,\lambda}=\Ind_P^G \chi_\lambda$. It follows easily by (\ref{eq-RadonNikonB}) that $\pi_\lambda$ is unitary if and only if $\lambda\in i\fa^*$. 

\begin{theorem}[Vogan-Wallach]\label{th-irredu} There exists a countable collection $\{p_n\}_{n}$ of non-zero holomorphic polynomials on $\fa_\C^*$ such that if $\lambda\in\fa_\C^*$ and $p_n(\lambda)\not= 0$, then $\pi_\lambda$ is irreducible. In particular, $\pi_\lambda$ is irreducible for almost all $\lambda\in \fa_\C^*$.
\end{theorem}
\begin{proof} This is Lemma 5.3 in \cite{VoganWallach}.\end{proof}

Replace $P$ by $\oP$, realize $\cB$ as $G/\oP$ and then induce the character determined by $\lambda$ from
$\oP$ to $G$. We denote the corresponding representation by $\pi_{\oP, \lambda}$.

\begin{lemma} Let $\pi^\theta_\lambda :=\pi_{P,\lambda}\circ \theta$. Then
$ \pi_{\oP,\lambda}=\pi^\theta_{-\lambda} $ as a representation of $G$ acting on
$L^2(\cB)$.
\end{lemma}
\begin{proof} This follows easily from Lemma \ref{le-BarAction}, equation (\ref{eq-APbar}) and $\rho_{\oP}=-\rho_P$:
\begin{align*}
[\pi_{\oP,\lambda} (x)f](b)&=a_{\oP}(x^{-1}b)^{-\lambda -\rho_{\oP}}f(x^{-1}\,\bar{\cdot}\, b)\\
&=a_P(\theta (x)^{-1}b)^{\lambda -\rho_P}f(\theta (x)^{-1} \cdot b)\\
&= [\pi_{P,-\lambda}(\theta (x))f](b)\, . \tag*{\qedhere}
\end{align*}
\end{proof}

For $f\in C^\infty (\cB)$ and $b=k\cdot b_o$ define the standard intertwining operator $J (\lambda )=J_{\oP|P} (\lambda )$ by 
\begin{equation}\label{defJ}
J(\lambda)f(b):=
\int_{\oN} a (\bar n)^{-\lambda -\rho}f ( k \bar n\cdot b_o)\, d\bar n
\end{equation}
whenever the integral exists. We note that this is well defined. Indeed, if $m\in L$ then $a(m\bar n m^{-1})= a(\bar n)$, and the Haar measure on $\oN$ is invariant under conjugation by $m$. Let $b=k\cdot b_o=k_1\cdot b_o$. Then $k=k_1m $ for some  $m\in L$. Thus
\begin{eqnarray*}
\int_{\oN} a(\bar n)^{-\lambda -\rho}f(k\bar n\cdot b_o)\, d\bar n &=&
\int_{\oN} a(\bar n)^{-\lambda -\rho}f(k_1m \bar n\cdot b_o)\, d\bar n\\
&=&
\int_{\oN} a(m \bar n m^{-1})^{-\lambda -\rho}f(k_1 (m \bar n m^{-1})\cdot b_o)\, d\bar n\\
&=&
\int_{\oN} a(\bar n)^{-\lambda -\rho}f(k_1 \bar n\cdot b_o)\, d\bar n \, .
\end{eqnarray*}

Recall that the space of smooth vectors for $\pi_\lambda$ is $C^\infty (\cB )$. In particular, it is independent of $\lambda$. If $X\in\fg$ then we define as usually $\pi^\infty_\lambda (X) : C^\infty (\cB)\to C^\infty (\cB)$ by
\[\lim_{t\to 0} \frac{\pi_\lambda (\exp (tX))f-f}{t}=\left.\dfrac{d}{dt}\right|_{t=0}\pi_\lambda (\exp (tX))f\]
where the derivative is taken in $L^2(\cB)$. Then $\pi_\lambda^\infty $ is a representation of $\fg$ in $C^\infty (\cB)$ and extends to a representation of the universal enveloping algebra $U(\fg )$. If $u\in U(\fg )$ then $\pi_\lambda^\infty (u) : C^\infty (\cB)\to C^\infty (\cB)$ is a differential operator.
If $\lambda \in\fa_\C^*$ write $\lambda =\lambda_R+i\lambda_I$ with $\lambda_R,\lambda_I\in\fa^*$.

\begin{theorem}[Vogan-Wallach] 
\label{thm:VW}
The following holds:
\begin{enumerate}
\item There exists a constant $a_P$ such that if
\[\lambda \in \fa_\C^* (a_P):=\{\mu \in \fa_\C^*\mid (\forall \alpha \in \Delta^+)\,\,
\langle \mu_R ,\ga \rangle \ge a_P\}\]
then the integral (\ref{defJ}) converges absolutely. Furthermore, there exists a constant $C>0$ such
that for all $\lambda \in \fa_\C^* (a_P)$ and $f \in C^\infty(\cB)$
\begin{equation}\label{jlambdafsup}
\|J(\lambda )f\|_\infty\le C \|f\|_\infty\, .
\end{equation}
\item If $f\in C^\infty (\cB)$ then
\[\fa_\C^* (a_P)\ni \lambda \mapsto J(\lambda )f\in C^\infty (\cB)\]
is continuous and holomorphic on the interior $\{\mu \in \fa_\C^*\mid (\forall \alpha \in \Delta^+)\,\,
\langle \mu_R ,\ga \rangle > a_P\}$.
\item There exists polynomial functions $b : \fa_\C^*\to \C$ and $D : \fa_\C^*\to U(\fg )^K$
such that if $\lambda \in \fa_\C^* (a_P)$ and $f\in C^\infty (\cB)$, then
\[b (\lambda )J(\lambda )f= J(\lambda +4\rho )\pi_{\lambda +4\rho} (D (\lambda ))f\, .\]
In particular $\lambda \mapsto J(\lambda )f$ extends to a meromorphic function on $\fa_\C^*$.
\item The operator $J (\lambda)$ intertwines $\pi_\lambda$ and $\pi_{-\lambda}^\theta$. 
Thus, if $x\in G$ and $f \in C^\infty(\cB)$, then
\begin{equation}\label{eq-Intertwining}
J (\lambda) (\pi_\lambda (x)f)=
\pi_{-\lambda}^\theta (x) J (\lambda)f\, .
\end{equation}
\end{enumerate}
\end{theorem}
\begin{proof} Referring to \cite{VoganWallach}, we have that (1) is  Lemma 1.2, (2) is Lemma 1.3, (3) is Theorem 1.5, and (4) is Lemma 1.2 for $\lambda \in \fa_\C^* (a_P)$. It extends to $\mathfrak a_\C^*$ by meromorphic continuation.
\end{proof}

Our situation is simpler than the general case considered in \cite{VoganWallach}, therefore--and for completeness--we describe the main idea behind the proof. Let
\begin{equation}\label{defaP}
a_P :=\max_{\alpha \in \Delta^+}\langle \alpha ,\rho\rangle > 0\, .
\end{equation}
According to Lemma A.3.3 in \cite{VoganWallach} we have $\langle \log a(\bar n),\alpha \rangle \ge 0$ for all $\alpha\in\Delta^+$. Thus for $\lambda \in \fa_\C^*(a_P)$
\[|a(\bar n )^{-\lambda -\rho}|\le a(\bar n)^{-2\rho}\]
and $\bar n \mapsto a(\bar n)^{-2\rho }$ is integrable. Thus, the integral defining $J(\lambda)$ exists and $\lambda \mapsto J(\lambda )f$ is holomorphic and also (\ref{jlambdafsup}) clearly follows. Note that (\ref{jlambdafsup}) says in particular that $J(\lambda )f\in C^\infty (\cB)$.

Part (2) follows by the above estimates (\ref{jlambdafsup}) that show that the integral defining $J(\lambda )f$ converges uniformly on any domain of the form $\fa_\C^* (a)$ with $a>a_P$.

The proof of (3) uses tensoring with finite dimensional representations of $G$. Another way of showing that $\lambda \mapsto J (\lambda )f$ extends to a meromorphic function on $\fa_\C^*$ is to use the
last part of (\ref{defJ}) and ideas from \cite{BD92, O87,OP} using the Bernstein polynomial for $\bar n \mapsto a(\bar n)^{-\lambda -\rho}$ and equation (\ref{defJ}).

We will show that $J (\lambda )$ is an intertwining operator after having proved Lemma \ref{lemma-JK} below.

\begin{lemma} Let $a_P$ be as in (\ref{defaP}) and let $\lambda \in \fa_\C^* (a_P)$. Let $1\le p\le \infty$ and let $f\in L^p(\cB)$. Then the integral (\ref{defJ}) exists and
$|J(\lambda )f(k)|\le \|f\|_p$.
In particular, $J(\lambda )f\in L^p(\cB)$ and $\|J(\lambda )f\|_p\le \|f\|_p$.
\end{lemma}
\begin{proof} The case $p=\infty$ is (\ref{jlambdafsup}). We can therefore assume that $p<\infty$. Let $q$ be such that $\frac{1}{p}+\frac{1}{q}=1$. Then $2\rho = \frac{2}{p}\rho +\frac{2}{q}\rho$. As $|a(\bar n)^{-\lambda -\rho}|\le a(\bar n)^{-2\rho}$ and $\int_{\oN} a(\bar n)^{-2\rho }\, d\bar n=1$, it follows that
\begin{eqnarray*}
\Big|\int_{\oN}a (\bar n)^{-\lambda -\rho}f(k k(\bar n )\cdot b_o)\, d\bar n\Big|
&\le & \int_{\oN} a(\bar n)^{-2\rho }|f(k k(\bar n)\cdot b_o )|\, d\bar n\\
&\leq& \left(\int_{\oN}
\left|a(\bar n)^{-\frac{2}{p}\rho}f(k k(\bar n)\cdot b_o )\right|^p \, d\bar n\right)^{1/p}\\
&=& \left( \int_{\oN }|f (k k(\bar n)\cdot b_o)|^p
a (\bar n)^{-2\rho }\, d\bar n\right)^{1/p}\\
&=& \left(\int_{\cB} |f(k b)|^p\, db\right)^{1/p}\\
&=&\|f\|_p\, .
\end{eqnarray*}
Here we used  (\ref{intKNbar}) to transfer the integral over $\oN$ to an integral over $\cB$.
\end{proof}

For $x\in \oN P$ write
\[x=\bar{n}_P(x) m_{P}(x)\alpha_P (x)n_{P}(x)\]
with $\bar{n}_P(x) \in \oN$, $m_{P}(x)\in M$, $\alpha_P (x)\in A$,
and $n_{P}(x)\in N$. Note that $x\mapsto \alpha_P (x)$ is right $MN$-invariant. We leave out the subscript ${}_P$ if it is clear which parabolic subgroup we are using.

\begin{lemma} 
\label{lemma-JK}
Assume that $\lambda\in\fa_\C^*(a_P)$. Let $f\in C^\infty (\cB)$. Then for $k \in K$
\[J (\lambda )f(k)=\int_K \alpha_P (k^{-1}h)^{\lambda-\rho}f(h)\, dh
=\int_{\cB} \alpha_P (k^{-1}b)^{\lambda-\rho} f(b)\, db\, .\]
In particular, $J(\lambda)$ is a convolution operator on $L^2(\cB)$ if $\lambda\in \fa_\C^* (a_P )$.
\end{lemma}\label{le-JK}
\begin{proof} Let $\bar{n}\in\oN$. Then
\begin{eqnarray*}
\bar n &=& k(\bar n) m (\bar n) a (\bar n) n(\bar n)\\
&=&\bar{n} (k(\bar n) ) m_{\oN}(k(\bar n))\alpha (k(\bar n))n_{\oN}(k(\bar n))
m (\bar n) a (\bar n) n(\bar n)\\
&=& \bar{n} (k(\bar n) ) m^\prime \alpha (k (\bar n )) a(\bar  n) n^\prime
\end{eqnarray*}
for some $m^\prime \in M$ and $n^\prime\in N$. Thus 
$\alpha (k (\bar n )) = a(\bar n )^{-1}$.
Assume that $\lambda \in \fa_\C^* (a_P)$. 
Using (\ref{intKNbar}) we get
\begin{eqnarray*}
\int_{\oN} a (\bar n)^{-\lambda -\rho } f(kk (\bar n)\cdot b_o )\, d\bar{n}
&=&\int_{\oN} \alpha (k (\bar n))^{\lambda -\rho} f(k k(\bar n)\cdot b_o)
a(\bar n)^{-2\rho}\, d\bar n\\
&=&\int_{\cB} \alpha (b)^{\lambda -\rho} f(kb)\, db\\
&=&\int_\cB \alpha (k^{-1}b)^{\lambda -\rho} f(b)\, db
\end{eqnarray*}
as the measure on $\cB$ is $K$-invariant.
\end{proof}

We now show that $J(\lambda )\pi_\lambda (x)=\pi_{-\lambda }(\theta (x))J(\lambda)$.
For this, let $x\in G$ and $\lambda \in\fa_\C^* (a_P)$. Then
\begin{equation}\label{eq-intertw1}
J(\lambda )(\pi_\lambda (x)f)(k)=\int_K \alpha (k^{-1}h)^{\lambda -\rho} a (x^{-1}h)^{-\lambda -\rho }f(k(x^{-1}h))\, dh\, .
\end{equation}
We have
\[\alpha (k^{-1}h)=\alpha (k^{-1}xx^{-1}h)=\alpha (k^{-1}xk(x^{-1}h))a(x^{-1}h)\, .\]
Hence we can rewrite (\ref{eq-intertw1}) as
\begin{eqnarray}
J(\lambda )(\pi_\lambda (x)f)(k)&=&\int_K \alpha (k^{-1}xk(x^{-1}h))^{\lambda -\rho} f(k(x^{-1}h) ) a(x^{-1}h)^{-2\rho}\, dh\nonumber \\
&=& \int_K \alpha (k^{-1} x h)^{\lambda -\rho} f(h)\, dh\nonumber\\
&=& \int_K \alpha ((x^{-1}k)^{-1} h)^{\lambda -\rho} f(h)\, dh\label{eq-intertw2}
\, .
\end{eqnarray}
Now use Lemma \ref{le-BarAction} to write
\[x^{-1}k =k_{\oP} (x^{-1}k)m a_{\oP} (x^{-1}k)\bar n
=k_{P}(\theta (x^{-1})k)m a_P (\theta (x^{-1})k)^{-1} \bar n\]
for some $m\in M$ and $n\in N$ and insert this into (\ref{eq-intertw2}) to get
\[J(\lambda )(\pi_\lambda (x)f)(k)=a (\theta (x)^{-1}k)^{\lambda -\rho }\int_K \alpha (k(\theta (x)^{-1}k)^{-1}h)^{\lambda -\rho} f(h)\, dh
=\pi_{-\lambda }(\theta (x))J (\lambda )f (k)\, . \]

As $J(\lambda )$ is $K$-intertwining and the only $K$-invariant functions in $L^2(\cB)$ are multiple of the constant function $1(x)=1$, it follows that there exists  a meromorphic function $c_P : \fa_\C^*\to \C$ such that $J(\lambda )1=c_ (\lambda)$. For $\lambda\in \fa_\C^*(a_P)$ the function $c_P(\lambda)$ is given by the converging integral
\begin{equation}\label{cP}
c_P(\lambda ):=J(\lambda )1=\int_{\oN}a(\bar n)^{-\lambda -\rho}\, d\bar n
= \int_K \alpha (k)^{\lambda -\rho}\, dk\, .
\end{equation}
\begin{theorem}\label{th-Inverse} Let the notation be as above. Then, as an identity of meromorphic functions,
\[ J(-\lambda )J (\lambda )=
J (\lambda)J(-\lambda )=c_P (\lambda )c_P(-\lambda)\,\id \]
Furthermore, if $\lambda\in i\fa^*$, then the normalized operator $c_P(\lambda )^{-1}J(\lambda )$ is unitary.
\end{theorem}
\begin{proof} We have
\[J(-\lambda )J(\lambda )\pi_\lambda (x)=
 J(-\lambda )\pi_{-\lambda }(\theta (x))J(\lambda )
 =\pi_\lambda (x)J(-\lambda )J(\lambda )\]
as $\theta^2=\id$. As $\pi_\lambda $ is irreducible for $\lambda$ in an open dense subset of $\fa_\C^*$ it follows that the $J(-\lambda )J(\lambda )$ is a scalar multiple of the identity on that set. As $J(-\lambda )J(\lambda )1= c_P(-\lambda )c_P(\lambda )$ it follows that $J(-\lambda )J(\lambda )=c(-\lambda )c(\lambda )\id$ as a meromorphic function. The same argument shows that $J(\lambda )J(-\lambda )=c(-\lambda )c(\lambda )\id$ and (1) follows. 
The unitarity of $c_P(\lambda )^{-1}\cA (\lambda )$ for $\lambda $ purely imaginary follows from
$[c_P(\lambda )^{-1} J(\lambda )]^*=c_P(-\lambda )^{-1} J(- \lambda  )$ as in this case $\overline{\lambda}=-\lambda$.
\end{proof}

\section{The Spectrum Generating Operator}\label{section3}
\noindent
We will now specialize the situation to the case where $G$ is simple and $P$ is a maximal parabolic subgroup. Thus $\dim  \fa=1$. Based on our application later on, we will also assume that $\Delta = \{\alpha ,-\alpha\}$. Hence $\fn$ is abelian. Finally, we will assume that $K$ is semisimple and that $(K,L)$ is a symmetric pair corresponding to an involution $\tau : K\to K$. This implies that each irreducible representation of $K$ that occurs in  $L^2(\cB)$ has multiplicity one.

Denote by $\widehat K$ the set of equivalence classes of irreducible representations of $K$. For simplicity we will often write $\delta \in \widehat K$ where $\delta$ is an irreducible representation of $K$, i.e., we identify the representation with its equivalence class. Denote by $V_\delta$ the finite dimensional Hilbert space on which $\delta $ acts and let $d (\delta ):=\dim V_\delta$. Let
\[V^L:=\{u\in V_\delta \mid (\forall k\in L)\,\, \delta (k)u = u\}\, .\]
Then, as $K/L$ is a symmetric space, either $\dim V_\delta^L=0$ or $\dim V_\delta^L=1$. We say that $(\delta ,V_\delta)$ is \textit{spherical}
if $V_\delta^L\not=\{0\}$. Let $\wKL$ denote the subset of $\widehat K$ of spherical representations. To describe $\wKL$ write $\fk = \fl\oplus \fq$ where $\fq=\{X\in \fk\mid \tau (X)=-X\}$. Let $\fb$ be a maximal abelian subspace of $\fq$. Denote by
$\Delta_\fk$ the set of roots of $\fb_\C$ in $\fk_\C$. Note that $\Delta_\fk\subset i\fb^*$. Let $\Delta^+_\fk$ be a set of positive roots and let $\rho_\fk$ be the corresponding half-sum of roots, counted with multiplicities. Denote by $W_\fk$ the  Weyl group of $\Delta_\fk$. Finally we let
\[\Lambda^+:=\{\mu \in i\fb^* \mid (\forall \alpha \in \Delta^+_\fk)\,\, \frac{\langle \mu ,\alpha \rangle}{\langle \alpha ,\alpha \rangle }\in \N_0\}\, .\]
Then the map $\pi \mapsto \,(\textrm{highest weight of $\pi$})$ defines an injective map of $\wKL$ into $\Lambda^+$. This map is bijective if and only if $\cB$ is simply connected. In general $\wKL$ is isomorphic to a sublattice $\LB$ of $\Lambda^+$. We refer to  \cite{GGA}, p.\ 535, and the discussion in \cite{OSMathScand}, in particular the appendix,
for more details. Let $\mu\in\LB$ and  denote by $\pi_\mu$ the corresponding spherical representation.
Let $V_\mu$ be the space on which $\pi_\mu$ acts and let $(\cdot,\cdot)_{V_\mu}$ the inner product on $V_\mu$ making $\pi_\mu$ unitary. 
We also fix a unit vector $e_\mu\in V_\mu^L$. Define $T_\mu : V_\mu \to L^2(\cB )$ by 
\[T_\mu (v)(x):= d (\mu )^{1/2} (v,\pi_\mu (x)e_\mu )_{V_\mu}\, .\]
Then $T_\mu $ is an isometric $K$-intertwining operator between $\pi_\mu$ and the left regular representation $L$ of $K$ on $L^2(\cB)$. Let $L^2_\mu (\cB):=\Im (T_\mu)$. Then
$L$ is an irreducible $K$-representation on $L^2(\cB)$ and
\[L^2(\cB)\simeq_K \bigoplus_{\mu\in\LB}L^2_\mu (\cB)\, .\]

\begin{lemma}\label{le-DefEta} 
For $\mu \in\LB$ let $J_\mu (\lambda ):= J(\lambda)|_{L^2_\mu (\cB)}$. There exists a meromorphic function $\eta_\mu :\fa_\C^*\to \C$ such that
\begin{enumerate}
\item $J_\mu (\lambda)=\eta_\mu (\lambda )\id_{L^2_\mu (\cB)}$.
\item $\eta_{0} (\lambda )= c_P (\lambda )$,
\item $\eta_{\mu }(\lambda )\eta_\mu (-\lambda )= c_P (\lambda )c_P (-\lambda )$,
\end{enumerate}
\end{lemma}
\begin{proof} 
$J (\lambda )$ is a $G$-intertwining operator between $\pi_\lambda$ and $\pi_{-\lambda}^\theta$ by Theorem \ref{thm:VW} (4).
As $K$-representations, $\pi_\lambda$ and $\pi_{-\lambda}^\theta$ agree with $L$. Hence $J(\lambda)$ is a $K$-intertwining operator of $L$ with itself. As each
$L_\mu^2(\cB)$ is irreducible and the multiplicity of $\pi_\mu $ in $L^2(\cB)$ is one, it follows by Schur's Lemma that we can define $\eta_\mu (\lambda )$ by (1). 
The claim now follows from our normalization
$J (\lambda) 1=c_P (\lambda )$ and Theorem \ref{th-Inverse}.
\end{proof}

We call $\{\eta_\mu (\lambda )\mid \mu\in\LB\}$ the $K$-spectrum of $J (\lambda )$.

Our aim is to determine the functions $\eta_\mu$ using the results from \cite{BOO}.  Fix $H_o\in \fa$ such that $\alpha (H_o)=1$. Fix a $G$-invariant $\R$-bilinear form
$\langle \, \cdot \, ,\, \cdot \,\rangle$ on $\fg$ such that
$\langle H_o, H_o\rangle = 1$. For $H\in\fa$ we have $\alpha (H)=
 \langle H, H_o\rangle $. Denote the  complex bilinear extension to
 $\fg_\C$ by the same symbol. Define a map $\omega :\fs_\C \to C^\infty (\cB)$ by
\[\omega (Y)(k):= \langle Y,\Ad (k)H_o\rangle \]
and note that
\[\omega (\Ad (h)Y)(k)=\langle \Ad (h)Y,\Ad (k)H_o\rangle = \langle Y, \Ad (h^{-1}k)H_o\rangle
=\omega (Y)(h^{-1}k)\, .\]
Thus $\omega $ is a $K$-intertwining operator.

Note that $\langle \, \cdot \, ,\,\cdot \, \rangle$ is negative definite on $\fk$.
Let $X_1,\ldots ,X_{\dim \fq}$ be an orthonormal basis of $\fq$ such that $X_1,\ldots ,X_p$, $p=\dim \fb$, is an orthonormal basis of $\fb$. Denote by $\Omega =- \sum_j X_j^2$ the corresponding positive definite Laplace operator on $\cB$. The action of $\Omega$ on functions on $\cB$ is--up to a sign--the same as that of the Casimir element of $\fk$. Thus $\Omega |_{L^2_\mu (\cB )} =\langle \mu +2\rho_\fk ,\mu\rangle \id$. For later reference we state this as a Lemma:
\begin{lemma}\label{le-eigenvalues} 
For $\mu \in \Lambda^+(\cB)$ let
\[\omega (\mu )=- \sum_{j=1}^{p} \mu (X_j)(\mu +2\rho_\fk)(X_j)=
- \sum_{j=1}^{p} \Big(\mu (X_j)^2 + 2\mu (X_j)\rho_\fk (X_j)\Big)\, .\]
Then
$
\Omega|_{L^2_{\mu}(\cB)}= \omega (\mu) \id$.
\end{lemma}

For $f\in C^\infty (\cB)$ denote by $M (f ) : L^2(\cB)\to L^2(\cB)$ the multiplication operator $g\mapsto fg$.
\begin{theorem} Let $Y\in \fs$. Then $[\Omega,M(\omega (Y))]=2\pi_0(Y)$.
\end{theorem}
\begin{proof} This is Theorem 2.3 in \cite{BOO}.\end{proof}

For $\mu \in\LB$ define $\Phi_\mu :  L^2_\mu (\cB)\otimes \fs_\C \to L^2(\cB)$ by
\[\Phi_\mu ( \varphi \otimes Y):= M(\omega (Y))\varphi\, .\]
Observe that for $k \in K$, $Y \in \fs_\C$ and $\varphi\in L^2_\mu(\cB)$ we have
\[ L_k\big( M(\omega(Y)\varphi\big)=\big(L_k\omega(Y)\big)(L_k\varphi)=M\big(\omega(\Ad(k)Y)\big)(L_k\varphi) \]
with $\Ad(k)Y \in \fs_\C$ and $L_k\varphi \in L^2_\mu(\cB)$. Hence $\Phi_\mu$ is $K$-equivariant and $\Im \Phi_\mu$ is $K$-invariant.
Define a finite subset $S(\mu)\subset \LB$ by
\[\Im \Phi_\mu \simeq_K \bigoplus_{\sigma \in S(\mu )} L^2_\sigma (\cB)\, .\]
Denote by $\pr_\sigma $ the orthogonal projection $L^2(\cB)\to L^2_\sigma (\cB)$.

\begin{lemma} Assume that $\mu\in \LB$. Let $\sigma\in S(\mu)$, $Y\in \fs_\C$ and $r\alpha \in\fa_\C^*$. Let
\begin{equation} 
\label{eq:omegamusigma}
\omega_{\sigma\mu} (Y) : = \pr_\sigma\circ M( \omega (Y))|_{L^2_\mu (\cB)} :
L^2_\mu (\cB)\to L^2_\sigma (\cB)\, .
\end{equation}
Then
\begin{equation}\label{eq-spectGen1}
\pr_\sigma\circ \pi_{r\alpha} (Y)|_{L^2_\mu (\cB)}=\frac{1}{2}(\omega (\sigma )-\omega (\mu )+2r )\omega_{\sigma\mu} (Y) \, .\end{equation}
\end{lemma}
\begin{proof} This is Corollary 2.6 in \cite{BOO}.\end{proof}
\begin{theorem}\label{th-3.6} Let $\mu \in \LB$, $\sigma \in S (\mu )$ and $\lambda=r\alpha\in\fa_\C^*$. Then
\begin{equation}\label{eq-spectGen3}
\frac{\eta_\sigma (\lambda )}{\eta_\mu (\lambda)} =\frac{2r -\omega (\sigma )+\omega (\mu )}{2r + \omega (\sigma )-\omega (\mu )}\, .\end{equation}
The $K$-spectrum $\{\eta_\mu (\lambda )\}_{\mu\in \LB}$ and hence $J (\lambda)$ is uniquely determined by
(\ref{eq-spectGen3}) and the normalization $\eta_{0}(\lambda )= c_P (\lambda )$.
\end{theorem}
\begin{proof} 
Applying  $J (\lambda)$ to (\ref{eq-spectGen1}) from the left, using that
$J (\lambda )$ commutes with $\pr_\sigma$ and that $J (\lambda ) \circ \pi_\lambda (Y)=\pi_{-\lambda}^\theta (Y)\circ J (\lambda) = - \pi_{-\lambda } (Y)\circ J (\lambda ) $, we obtain:
\[(\omega (\sigma )-\omega (\mu )+2r)\eta_\sigma (\lambda )\omega_{\sigma \mu}(Y)
=-(\omega (\sigma )-\omega (\mu )-2r)\eta_\mu (\lambda )\omega_{\sigma\mu}(Y)\, .\]
As $\omega_{\sigma\delta}(Y)$ is non-zero for generic $\lambda$ it can be canceled out.  The last statement follows from the fact that
$\pi_\lambda$ is irreducible for generic $\lambda$, hence iterated application of  (\ref{eq:omegamusigma}) will in the end reach all $K$-types starting from the trivial $K$-type.
\end{proof}

\section{The case of the Grassmann Manifolds}\label{se-Grassman}
\noindent
From now on $\K $ denotes the one of the fields $\R$ or $\C$, or the skew field $\H$ of quaternions. We let $G=\SL (n+1,\K)$.
For discussion about the determinant function on the space of  quaternionic matrices see
\cite{As96,GRW}. We will use that if $x\in \rM (n+1,\K)$ then $|\det_\R x|=|\det_\K x|^d$, where $d=\dim_\R \K$ is equal to $1,2, 4$ for $\K=\R,\C,\H$, respectively. Denote by $\mathrm{G}_{p}(\K)$ the Grassmann manifold of $k$-dimensional subspaces of $\K^{n+1}$. 
Note that $\rG_{p}(\K )\simeq \rG_{n+1-p}(\K )$. So we can assume without loss of generality that $2p \leq n+1$. We set $q=n+1-p$. Hence $q \geq p$.

In this section we show that $\mathrm{G}_{p}(\K)=G /P = K/L$, where $P$ is a maximal parabolic subgroup of $G$. Furthermore $K/L$ is a symmetric space. Therefore the results from the previous sections apply to this case. We introduce the $\CosL$-transform for $K/L$ 
and then show that it agrees with the intertwining operator $J(\lambda)$.  In the next section we then apply the spectrum generating operator to determine the $K$-spectrum.

Define an invariant $\R$-bilinear form on $\fg$ by
\[\langle X,Y\rangle :=\frac{n+1}{pq}\Re(\Tr (XY))\, .\]
Denote by $z\mapsto \bar z$ the conjugation in $\K$ and for $x=(x_{\nu\mu} )\in \rM (n+1,\K)$ let
$\overline{x}:= (\overline{x}_{\nu \mu})$ and $x^*:=(\overline{x}_{\mu\nu})=\overline{x}^t$. The homomorphism
$\theta  : G \to G$, $x\mapsto (x^{-1})^*$ is a  Cartan involution on $G$. The corresponding Cartan involution on $\fg$ is $\theta (X)=-X^*$. We have
\begin{eqnarray*}
K&=&\SU(n+1,\K):=\{x\in G\mid x^*=x^{-1}\}\\
\fk &=& \{X\in \rM (n+1,\K )\mid \text{$X^*=-X$ and $\Tr (X)=0$}\}\\
\fs&=&\{X\in\rM (n+1,\K)\mid \text{$X^*=X$ and $\Tr (X)=0$}\}
\end{eqnarray*}
In the notation in \cite{Sig} we have $G=\SL (n+1,\R)$ and $K=\SO (n+1)$ for $\K=\R$,  $G=\SL (n+1,\C)$ and $K=\SU (n+1)$ for $\K=\C$, and $G=\SU^* (2(n+1))$ and $K=\Sp (n+1)$ for $\K=\H$.

For $k\in \N$ denote by $\rI_k$ the $k\times k$ identity matrix. We write matrices as blocks
\[\begin{pmatrix} X & Y \\ W & Z\end{pmatrix}\,, \qquad
\text{$X\in \rM_{p\times p}(\K)$\,, $Y\in \rM_{p\times q}(\K)$\,,
$W\in \rM_{q\times p}(\K)$ and $Z\in \rM_{q\times q}(\K )$\,.}\]
Let
\[H_o=\begin{pmatrix} \frac{q}{n+1}\rI_p & 0\cr 0 &
-\frac{p}{n+1}\rI_q\end{pmatrix}\in \fs\, ,\]
and define
$\fa :=\R H_o$, and $\fm=\{X\in \fz_\fg (\fa)\mid
\langle X,H_o\rangle =0\}$. Then
$\fz (\fm ) \cap \fs =\fa$. We have $\Delta =\{\alpha, -\alpha\}$ where $\alpha (H_o)=1$.
As $\langle H_o,H_o\rangle =1$ it follows that $\langle \,\cdot \, , \,\cdot \, \rangle$ agrees
with the invariant form from last section.

Let $\Delta^+ =\{\alpha \}$.
Then
\begin{eqnarray*}
\fm\oplus \fa & =&\left\{\left. M(X,Y)=\begin{pmatrix}
X & 0\cr
0 & Y\end{pmatrix}\, \right|\, \begin{matrix}
X\in \rM(p,\K )\\ Y\in \rM(q,\K)\end{matrix} \text{ and } \Tr X + \Tr Y=0\right\}\\
&=&\fs (\fg\fl (p,\K)\times \fg\fl (q,\K))\, ,\\
\fn &=&\left\{\left. N (X):=\begin{pmatrix}
0_{pp} & X \\ 0_{qp}& 0_{qq}
\end{pmatrix}
\, \right|\,
X\in M(p\times q,\K)\right\}
\simeq M(p\times q,\K)\\
MA&=&\left\{\left. m(a,b):=\begin{pmatrix}
a & 0\cr
0 & b\end{pmatrix}\, \right|\,
\begin{matrix} a\in \GL(p,\K )\\ b\in \GL(q, \K)\end{matrix}
\text{ and } \det a \,\det  b=1\right\}\\
&=&\rS (\GL (p,\K)\times \GL (q,\K))\\
N&=&\left\{\left. n(X)=\begin{pmatrix} \rI_p & X\\
0_{qp} & \rI_q\end{pmatrix}\, \right|\, X\in \rM (p\times q,\K)\right\}\\
P&=&  \left\{\left. p(a,b;X):=\begin{pmatrix}
a & X\cr
0 & b\end{pmatrix}\, \right|\,
\begin{matrix} a\in \GL(p,\K ) \\
b\in \GL(q, \K)\end{matrix}\,,\; \det a \,\det  b=1,\, X\in \rM (p\times q,\K)\right\}\\
L&=& \rS (\rU(p,\K)\times \rU(q,\K))\, .
\end{eqnarray*}

Let $e_1,\ldots ,e_{n+1}$ be the standard basis  for $\K^{n+1}$ and let $b_o=\K e_1\oplus \ldots \oplus \K e_p\in
\Gr_{p} (\K )$. Then $\Gr_{p}  (\K )=K\cdot b_o \simeq K/L = G/P$. Thus we can introducce the notation $\cB =\Gr_p(\K)$, which agrees with the one from the previous sections. Note that $L=\rS (\rO (p)\times \rO (q))$ for $\K =\R$, $L=\rS (\rU (p )\times \rU (q))$ for $\K=\C$, and $L=\Sp (p)\times\Sp (q)$ for $\K=\H$. In particular $\cB$ is simply connected for $\K =\C$ and $\H$.

Applying $\theta$ to $\fn$, we see that
\[
\bar \fn=\left\{\bar N(Y)=\left. \begin{pmatrix} 0_{pp} & 0_{pq}\\ Y & 0_{qq}\end{pmatrix}\,\right|\,
Y\in \rM_{q\times p}(\K )\right\}\simeq  \rM (q\times p, \K)
\]
and
\[\oN =\left\{\left. \bar n (Y) =\begin{pmatrix} \rI_p & 0_{pq}\\
Y & I_{q}\end{pmatrix}\, \right|\, Y\in \rM (q\times p, \K)\right\}\, .\]
The action of $MA$ on $\fn$, respectively $\bar \fn$, is given by
\[m(a,b)N (X) m(a^{-1},b^{-1})=N (aXb^{-1}) \text{ and } m(a,b)\bar N(Y) m(a^{-1},b^{-1})=
\bar N (bYa^{-1})\, .\]

\begin{lemma} $\displaystyle{\oN P
=\left\{\left. \begin{pmatrix} A & B\\
C& D\end{pmatrix}\in G\, \right|\, A\in \GL (p,\K)\right\}}\, .$
\end{lemma}
\begin{proof} A simple calculation shows that
\begin{equation}\label{eq-alphaFactor}
\bar n (Y) m(a,b)n (X) =\begin{pmatrix} a &aX\\ Ya & YaX+b\end{pmatrix}\,,
\end{equation}
proving one of the inclusions. On the other hand, let $x=\begin{pmatrix} A & B\\
C&D\end{pmatrix} $ be such that $A$ is invertible. Define
$a = A$, $X =a^{-1}B$, $Y=Ca^{-1}$, and $b=D-Ca^{-1}B$. Then
$m(a,b)\in MA$ and $\bar n (Y ) m(a,b)n (X)=x$.
\end{proof}

Let $b\in \Gr_p (\K)$ and view $b$ as a $dp$-dimensional real vector space. Fix a convex subset $E\subset b$ containing the zero vector, such that the volume of $E$ is one.
For $c\in \Gr_p (\K )$ let $P_c : \K^n\to c$ denote the orthogonal projection onto $c$. Define
\begin{equation}\label{defCos}
|\Cos (b,c)|:=\Vol (P_c(E))^{1/d}
\end{equation}
The definition is independent of $E$ as we will see in a moment. Recall, if $x: b\to b$ is linear then
\begin{equation}\label{eq-Volume}
\Vol (x(E))= |\det_\R (x)|= |\det (x)|^{d}.
\end{equation}
As the elements in $K$ acts as orthogonal transformations it follows that $|\Cos (k\cdot b,c)|=|\Cos (b,k^{-1}c)|$. In particular, if $b=k\cdot b_o$ and
$c=h\cdot b_o$ with $k,h \in K$, then
\begin{equation}\label{cos3}
|\Cos (b,c)|=|\Cos (h^{-1}k\cdot b_o,b_o)|=\Vol (P_{b_o} (h^{-1}k E))^{1/d}\, .
\end{equation}

As a motivation for the definition assume that $p=1$ and $\K=\R$.
If $b=\R x, c=\R y\in \Gr_{1}(\R)$, then
\[|\Cos (b,c) |=\frac{|(x,y)|}{\|x\|\|y\|}=|\cos (\sphericalangle (x,y))|\]
where $\sphericalangle (x,y)$ denotes the angle between $x$ and $y$.

For reasons that will become transparent in a moment, we identify $\fa_\C^*$ with $\C$ by
\begin{equation}\label{a=C}
\lambda \mapsto \frac{n+1}{pq}\lambda (H_o)
\text{ with inverse } z\mapsto z\, \frac{pq}{n+1}\alpha\, .
\end{equation}
We have $\dim_\R \fn= dpq$. Hence
\[\rho = \frac{dpq}{2}\alpha \longleftrightarrow \frac{d(n+1)}{2}\in\R \, .\]
Then $\rho$ is an integer if and only if $\K=\C, \H$ or $\K=\R$ and $n$ is odd.
For a complex number $\lambda$, we write from now on $\alpha_P(x)^\lambda$ instead of
$\alpha_P(x)^{\lambda \frac{pq}{n+1}\alpha}$.

\begin{theorem}\label{theCos}  
Fix a convex subset $E\subset b_o$ containing zero such that $\Vol (E)=1$. For $x\in G$ and  $\lambda\in\C$ we have
\begin{equation}\label{eqAlpha}
\alpha_P(x)^{\lambda } =|\Vol (P_{b_o}(x(E)))|^{\lambda/d } =|\Cos (x\cdot b_o, b_o)|^{\lambda}\, .
\end{equation}
In particular, if $b=k\cdot b_o$ and $c=h\cdot b_o\in \cB$, then
\begin{equation}\label{eq-alphaCos}
\alpha (h^{-1}k)^\lambda = |\Cos (b,c)|^{\lambda  }\, .
\end{equation}
\end{theorem}
\begin{proof} For $z=(z_1,\ldots ,z_p)^t\in \K^p$ we have
\[\begin{pmatrix} A & B\\ C & D\end{pmatrix}
\begin{pmatrix} z\\ 0\end{pmatrix}=\begin{pmatrix} Az\\ Cz\end{pmatrix}\, .\]
Thus
\[P_{b_o}\left(\begin{pmatrix} A & B\\ C & D\end{pmatrix}
\begin{pmatrix} z\\ 0\end{pmatrix}\right)=\begin{pmatrix} Az\\ 0 \end{pmatrix}\, .\]
Thus, by (\ref{cos3}),
\begin{equation}\label{eq-Cosl1}
\left|\Cos \left(\begin{pmatrix} A & B\\ C & D\end{pmatrix}\cdot b_o,b_o\right)\right|^{\lambda}
=|\det_\R A|^{\lambda/d}\, .
\end{equation}

Let $\bar n (Y) m(a,b)n (X)\in \oN P$.
Write $m(a,b)=m_o\exp (x H_0)$ with $m_o\in M$ and  $x\in \R$. Then
\[|\det_\R a |^\lambda =\exp \left(\lambda \frac{xdpq}{n+1}\right)
= \exp (xH_o)^{\lambda \frac{dpq}{n+1}\alpha}=\alpha_P\big(\exp(xH_0)\big)^{d\lambda}\, .\]
Thus
\begin{equation}\label{alphaAndDet}
\alpha_P(\bar n (Y) m(a,b)n(X))^\lambda =
|\det_\R (a)|^{\lambda /d }
=|\det a |^\lambda
\end{equation}
and (\ref{eqAlpha}) follows now from (\ref{eq-Cosl1}).
\end{proof}

This shows in particular that the definition of $\Cos$ is independent of the convex set $E$.
\begin{lemma}\label{le-alphaSym} 
Let $k\in K$. Then $\alpha_P(k)=\alpha_P(k^{-1})$.
\end{lemma}
\begin{proof} Write $k=\bar n m a n$. Then, as $\theta (k)=k$ and $\theta (a)^{-1}=a$, we have
$k^{-1}=\theta (k)^{-1}=\theta (n)^{-1} \theta (m)^{-1}a \theta (\bar n)^{-1}$
and the claim follows.
\end{proof}
\begin{corollary} Let $b,c\in \cB$. Then $|\Cos (b,c)|^\lambda =|\Cos (c,b)|^\lambda$.
\end{corollary}
\begin{proof} Write $b=k\cdot b_o$ and $c=h\cdot b_o$. By (\ref{eq-alphaCos}) and Lemma \ref{le-alphaSym}
\[|\Cos (b,c)|^\lambda =\alpha (h^{-1}k)^{\lambda}=\alpha (k^{-1}h)^\lambda =|\Cos (c,b)|^\lambda\,,\]
which proves the Corollary.
\end{proof}

The $\CosL$-transform $\CL$ is defined for $\Re \lambda >\rho$ by
\begin{equation}\label{def-CosL}
\CL (f)(c):= \int_{\cB} |\CosL (b,c)|^{\lambda -\rho}f(c)\, dc
=\int_K \alpha (k^{-1}h)^{\lambda -\rho} f(h\cdot b_o)\, dk\, \quad c=k\cdot b_o
\end{equation}
where we have used Theorem \ref{theCos}. Thus, $\CL$ is nothing but the standard intertwining operator $J(\lambda )$.
We would like to point out, that the proof of the intertwining property of $J(\lambda )$ shows that $\CL$ is an intertwining operator without the reference to the general theory. But the general theory described in Section \ref{section2} gives us the following theorem for free and allows us to apply the results from \cite{BOO} described in Section \ref{section3}.
\begin{theorem} The following holds:
\begin{enumerate} 
\item The $\CosL$-transform extends to a meromorphic family of intertwining operators $\CL  : C^\infty (\cB)\to C^\infty (\cB)$.
\item Let $c_P(\lambda ):= \CL (1)$. Then $c_P : \C \to \C$ is meromorphic and
\[\mathcal{C}^{-\lambda }\circ \CL = \CL \circ \mathcal{C}^{-\lambda}=c_P(-\lambda )c_P (\lambda)\id\, .\]
\item If $\lambda \in i\R$, then $\frac{1}{c_P (\lambda )}\CL$ is unitary.
\end{enumerate}
\end{theorem}

\section{The $K$-spectrum of the $\CosL$-Transform}\label{sec:Kspectrum}
\noindent
In this section we determine the $K$-spectrum of $J(\lambda )$ for the Grassmann manifolds 
$\rG_{p}(\K)=\SU(n+1,\K)/\rS (\rU(p,\K)\times \rU (q,\K))\simeq \SL (n+1,\K)/P$. This then determines the $K$-spectrum of the $\cos^\lambda$ transform. As before we always assume that $p\le q$ and $p+q=n+1$. We keep the notation from the previous sections.

The following is well known.

\begin{lemma}
The space $\cB$ is a symmetric space and the corresponding involution is
\[\tau (x)=\begin{pmatrix} \rI_p & 0 \\ 0 & -\rI_q\end{pmatrix}x\begin{pmatrix} \rI_p & 0 \\ 0 & -\rI_q\end{pmatrix}
=\begin{pmatrix} A & - B\\ - C & D\end{pmatrix} \text{ for } x=\begin{pmatrix} A & B\\ C & D\end{pmatrix}\]
and $K^\tau = L$.
\end{lemma}

The derived involution is given by the same formula. Let $\mathfrak l$ denote the Lie algebra of $L$. Then $\mathfrak k=\mathfrak{l}\oplus \mathfrak{q}$ where
\[\fq=\left\{\left. Q (X)=\begin{pmatrix} 0_{pp} & X\\ -X^* & 0_{qq}\end{pmatrix}\, \right|\, X\in \rM (p\times q,\K)
\right\}\, .\]

Let $E^{(r,s)}_{\nu,\mu}=(\delta_{i\nu}\delta_{j\mu})_{i,j}$ denote the matrix in $\rM(r\times s,\K)$
with all entries equal to $0$ but the $(\nu,\mu)$-th which is equal to $1$.
For $\mathbf{t}=(t_1,\ldots ,t_p)^t\in \R^p$ let 
\begin{align*}
X (\mathbf{t})&=-\sum_{j=1}^p t_j E_{q-p+j,j}^{(p,q)}\,\in \rM(p\times q,\K)\,,\\
Y(\mathbf{t})&= Q(X(\mathbf{t}))         \in \fq\,\\
\diag(\mathbf{t})&= \sum_{j=1}^p t_jE_{j,j}\, \in \rM(p\times p,\K)\,.
\end{align*}
Then $\fb = \{ Y (\mathbf{t})\mid \mathbf{t}\in\R^p\}\simeq \R^p$
is a maximal abelian subspace of $\fq$.

To describe the set $\wKL$ we note first that if $\K=\R$, then $\cB$ is not simply connected. In all other cases $\cB$ is simply connected. Define $\epsilon_j(Y(\mathbf{t})):=t_j$.   We will identify the element $\lambda=\sum_{j=1}^p\lambda_j \epsilon_j$ with the corresponding vector $\mathbf{\lambda
}=(\lambda_1,\ldots ,\lambda_p)$. 

\begin{lemma}
We have
\[\Delta_\fk=\{\pm \epsilon_i\pm \epsilon_j \,(1\le i\not= j\le p\,, \pm \text{independent}), \, \pm \epsilon_i \, (1\le i\le p)\,,\pm 2\epsilon_i \,(1\le i\le p)\, \}\]
with multiplicities respectively $d$ (and not there in case $p=1$), $d(q-p)$ (and not there in case $p=q$) and $d-1$ (and absent if $d=1$).
\end{lemma}
 
\begin{proof} This follows from \cite{Sig}: the table on page 518, the description of the simple root systems on page 462 ff. and the Satake diagrams on page 532--533.
\end{proof}

\begin{lemma}\label{le-Omegagrassmanian} Let the notation be as above.
\begin{enumerate}
\item 
If $\K=\R$ and $p=q$, then $$\LB =\{\mu=\sum_{j=1}^pm_j\e_j\mid m_j\in 2\N_0 \text{and} m_1\ge \ldots \ge m_{p-1}\ge |m_p| \;\}\,.$$
In all other cases, $$\LB =\{\mu=\sum_{j=1}^pm_j\e_j\mid m_j\in 2\N_0 \text{and} m_1\ge \ldots \ge m_p\ge 0 \;\}\,.$$
\item $\rho_\fk =\sum_{j=1}^p [\rho -d(j-1)-1]\epsilon_j$.
\item Let $\mu \in \LB$. Then
$\displaystyle{ \omega (\mu )=\frac{pq}{2(n+1)}\sum_{j=1}^p \Big(m_j^2 +2m_j (\rho -d(j-1)-1)\Big)}$.
\end{enumerate}
\end{lemma}
\begin{proof} (1) is well known and (2) follows by a simple calculation using the form of $\Delta_\fk^+$ and the multiplicities. For (3)  let $\w X_1,\ldots, \w X_p\in\fb$ be so that $\e_j (\w X_k)=i\delta_{jk}$.
Then $-\langle \w X_k,\w X_k\rangle = 2(n+1)/pq$. Thus the
$X_k=\sqrt{pq/[2(n+1)]}\, \w X_k$ form an orthonormal basis for $\fb$ and $-\epsilon_j(X_k)^2=pq/[2(n+1)] \delta_{jk}$.
The claim now follows from Lemma \ref{le-eigenvalues}.
\end{proof}

\begin{lemma} Let $\mu = (m_1,\ldots ,m_p)\in \LB$. Then
 \[S (\mu )=\{\mathbf{\mu} \pm 2\epsilon_j\mid j=1,\ldots , p \}\cap \LB\, .\]
These representations occur with multiplicity one.
\end{lemma}

\begin{proof} A simple matrix calculation shows that the weights of $\fh$ in $\fs_\C$ are $\pm (\epsilon_i +\epsilon_j)$ for $i,j=1,\ldots ,p$. Next we note that $\mu \pm (\epsilon_i+\epsilon_j)$, $j\not= i$, is not in $\LB$ as the $i^{{\rm th}}$ and $j^{{\rm th}}$ entries are not even.  It follows now from \cite{Kumar88} and the fact that each submodule in $V_\mu\otimes \fs_\C$ that occur in $L^2(\cB )$ is spherical that $\sigma\in S (\mu )$ if and only if there exists $j=1,\ldots ,p$ and $w\in W_\fk$ such that $w(\mu \pm 2\epsilon_j)$ is dominant and $\sigma =w(\mu \pm 2\epsilon_j)$. Assume that $\mu +2\epsilon_j$ is not dominant. Then $m_{j-1}<m_j+2$. Hence $m_j=m_{j-1}$. Applying the transposition $(j-1,j)$ shows that $\mu+2\epsilon_j$ is $W_\fk$-conjugate to $\mu+2\epsilon_{j-1}$. Repeating as long as necessary shows that there exists a $k$, $1\le k<j$ such that $\mu +2\epsilon_k$ is dominant and $W_\fk$-conjugate to $\mu+2\epsilon_j$. The case $\mu-2\epsilon_j$ not dominant is handled in similar way.
\end{proof}
\begin{lemma}\label{le-eta1}
Let $\mu =(m_1,\ldots ,m_p)\in \LB$ and $\lambda\in\C\simeq \fa_\C^*$. Then
\begin{equation}\label{eq-eta1}
\frac{\eta_{\mu +2\epsilon_j}(\lambda )}{\eta_\mu (\lambda )}=\frac{\lambda -m_j - \rho +d(j-1)}{\lambda  + m_j + \rho -d(j-1)}\, .
\end{equation}
\end{lemma}
\begin{proof} By Lemma \ref{le-Omegagrassmanian} we get
\[
\omega (\mu +2\epsilon_j)-\omega (\mu)=\frac{2pq}{n+1}\left(m_j +\rho -d(j-1)\right)\, .\]
The claim now follows from (\ref{eq-spectGen3}) and our identification $\C\simeq \fa_\C^*$, $z\mapsto z\frac{pq}{n+1}\alpha$.
\end{proof}

Introduce the Siegel or Gindikin $\Gamma$-function $\Gamma_{p,d}: \C^p\to \C$ by
\begin{equation}\label{def-Gamma}
\Gamma_{p,d} ( \lambda )=\prod_{j=1}^p \Gamma (\lambda_j -\frac{d}{2}(j-1))\, .
\end{equation}
Then $\Gamma (p,d;\lambda )$ is meromorphic with singularities on the union of the hyperplanes
\[\cH (d,j,n)=\{\lambda \in\C^p\mid \lambda_j =-n+\frac{d}{2}(j-1)\}\, ,\quad j=1,\ldots ,p,\,\, n\in \N_0\, .\]
In the following, if $\lambda\in\C$, we will also indicate by $\lambda$ the vector $(\lambda, \ldots ,\lambda)\in\C^p$. It will be clear from the context if we are using $\lambda$ as a complex number or as a vector.
\begin{theorem}\label{th-eta2} 
Let $\lambda\in\C$. Then there exists a meromorphic function $G(\lambda )$ such that for all $\mu\in\LB$ we have
\begin{equation}\label{eq-eta2}
\eta_\mu (\lambda )=(-1)^{|\mu|/2} G(\lambda) \frac{\Gamma_{p,d} \left(\frac{1}{2}(-\lambda+\rho+\mu)\right)}
{\Gamma_{p,d} \left(\frac{1}{2}(\lambda+\rho+\mu)\right)}\,, 
\end{equation}
where $|\mu|=\sum_{j=1}^p m_j$.
\end{theorem}
\begin{proof} Denote the right hand side of (\ref{eq-eta2}) by $\sigma_\mu (\lambda)$. Using $\Gamma (z+1)=z\Gamma (z)$, we get
\[\sigma_{\mu + 2\epsilon_j}(\lambda )=\sigma_\mu (\lambda )\, \frac{\lambda - m_j -\rho + d(j-1)}{\lambda + m_j + \rho - d(j-1)}\, .\]
Now Lemma \ref{le-eta1} implies that there exists a meromorphic function $G(\lambda )$ with the requires property.
\end{proof}

The function $G(\lambda)$ will be explicitly computed in (\ref{eq:G}) below, by comparing the formula for $\eta_0(\lambda )=c_P(\lambda)$.
For this, we note the following result which 
can be found 
in \cite{AAR}, Theorem 8.1.1.
\begin{lemma}[Selberg's integral] \label{le:SelbergInt}
Let $\gamma_1,\gamma_2, \alpha \in \C$ with $\Re\gamma_1,\Re\gamma_2, \Re\alpha >0$. Then
\begin{eqnarray*}
\lefteqn{\int_0^1\!\!\cdots \!\!\int_0^1\prod_{j=1}^p t_j^{\gamma_1-1}(1-t_j)^{\gamma_2-1}\prod_{1\le i< j\le p}|t_i-t_j|^{2\alpha}\, dt_1\ldots dt_p=}\\
&=&\prod_{j=1}^p\frac{\Gamma(\alpha j +1)\Gamma (\alpha (j-1)+\gamma_1)\Gamma (\alpha (j-1)+\gamma_2)}{\Gamma (\alpha +1)\Gamma (\alpha (p+j-2)+\gamma_1+\gamma_2)}\, .
\end{eqnarray*}
\end{lemma}
To complete the calculation we need to evaluate $\alpha_P (\exp (Y(\mathbf{t})))$. We have
\begin{equation}\label{eq-CosL1}
\exp Y(\mathbf{t})=\begin{pmatrix} \diag(\cos t_1,\ldots ,\cos t_p) & 0& -\diag(\sin t_1,\ldots ,\sin t_p)\\
0 & \rI_{q-p}& 0\\
\diag(\sin t_1,\ldots ,\sin t_p) & 0 & \diag(\cos t_1,\ldots ,\cos t_p)\end{pmatrix}
\end{equation}
where the matrix $\rI_{q-p}=\rI_{n+1-2p}$ does not occur in case $p=(n+1)/2$.

For the following see also \cite{Zhang09}, Lemma 3.2:
\begin{lemma}\label{le-Cosl3} Let $\mathbf{t}\in\R^p$ and $\lambda\in \C$. Then
$\displaystyle \alpha_P(\exp Y(\mathbf{t}))^\lambda =\prod_{j=1}^p|\cos (t_j)|^{\lambda }$.
\end{lemma}
\begin{proof} This follows from (\ref{eq-CosL1}), (\ref{alphaAndDet}) and (\ref{eq-alphaFactor}).
\end{proof}

\begin{theorem} \label{th-cPlambda}Let $c_P(\lambda )=\int_K \alpha_P(k)^{\lambda -\rho}\,dk=\eta_0(\lambda )=J(\lambda ) 1$. Then
\begin{equation}\label{eq:cP}
c_P(\lambda )= \frac{\Gamma_{p,d}\left(d(n+1)/2\right)}{\Gamma_{p,d} \left(dp/2\right)}\frac{\Gamma_{p,d} \left(\frac{1}{2}(\lambda -\rho + dp)\right)}{ \Gamma_{p,d}\left(\frac{1}{2}(\lambda+\rho )\right)}\, .
\end{equation}
\end{theorem}
\begin{proof} In the following $c$ will denote a positive constant that might change from one line to another. Its value will then be determined at the end. 
Let $\gamma_1=(\lambda -\rho +d)/2$,  $\gamma_2=d(q-p+1)/2$, $Q=[0,1]^p$, and $Q_\pi =[0,\pi]^p$. Using that $\sin (2u)=2 \cos (u)\sin (u)$ and $|\sin (u-v)\sin (u+v)|= |\cos^2 (u)-\cos^2 (v)|$, we get by the substitution $u_j =\cos^2(t_j)$:
\begin{eqnarray*}
c_P(\lambda )&=&c\int_{Q_\pi}\prod_{j=1}^p|\cos t_j|^{2\gamma_1-1}\sin (t_j)^{2\gamma_2-1}\prod_{i<j}|\cos^2(t_i)-\cos^2(t_j)|^d\,  dt_1\ldots dt_p\\
&=& c\int_{Q_\pi}\prod_{j=1}^p|\cos t_j|^{2\gamma_1-1} (1-\cos^2t_j)^{\gamma_2-1}\sin (t_j)\prod_{i<j}|\cos^2(t_i)-\cos^2(t_j)|^d\,  dt_1\ldots dt_p\\
&=&c \int_Q  \prod_{j=1}^pu_j^{\gamma_1-1 }(1-u_j)^{\gamma_2-1 }\prod_{i<j}|u_i-u_j|^d\, du_1\ldots du_p\\
&=&c\prod_{j=1}^p  \frac{\Gamma \left(\frac{1}{2}(\lambda -\rho + dj)\right)}{ \Gamma\left(\frac{1}{2}(\lambda -\rho + d(q+j))\right)}\\
&=&c\prod_{j=1}^p \frac{ \Gamma \left(\frac{1}{2}(\lambda -\rho +d p - d(j-1))\right)}{ \Gamma\left(\frac{1}{2}(\lambda +\rho -d(j-1))\right)} \, .
\end{eqnarray*}
In the last equation we used that $d(q+j)=d(n+1-p+j)=2\rho -dp +dj$ and that for a complex number $z$
\begin{equation}\label{eq:djto}
\prod_{j=1}^p \Gamma\left(\frac{1}{2}(z +dj)\right)=
\prod_{j=1}^p \Gamma\left(\frac{1}{2}(z+d p -d(j-1))\right)\, .
\end{equation}

Now evaluate $c$ by using that $c_P(\rho) =1$ and recalling that $\rho =d(n+1)/2$.
\end{proof}
\begin{theorem}\label{th:etaMu} The $K$-spectrum the $\CosL$-transform  is given by:
\begin{equation}\label{eq-eta3}
\eta_\mu (\lambda)=(-1)^{|\mu|/2} \; \frac{\Gamma_{p,d}\left(d(n+1)/2\right)}{\Gamma_{p,d} \left(dp/2\right)}\, \frac{\Gamma_{p,d}\left(\frac{1}{2}(\lambda - \rho +dp)\right)\Gamma_{p,d}\left(\frac{1}{2}(-\lambda+\rho+\mu)\right)}{\Gamma_{p,q} \left(\frac{1}{2}(-\lambda+\rho)\right)
\Gamma_{p,d} \left(\frac{1}{2}(\lambda+\rho+\mu)\right)}\, .
\end{equation}
\end{theorem}
\begin{proof} By taking $\mu =0$ in (\ref{eq-eta2}) and using $\eta_0=c_P$, we get
\begin{equation}\label{eq:G}
G(\lambda )=\frac{\Gamma_{p,d}\left(\frac{1}{2}(\lambda+\rho)\right)}{\Gamma_{p,d} \left(\frac{1}{2}(-\lambda+\rho)\right)}\, c_P(\lambda)\, . 
\end{equation}
The claim then follows from (\ref{eq-eta2}) and (\ref{eq:cP}).
\end{proof}

\begin{remark}
In Section \ref{section2} we referred to the result of Vogan and Wallach on the meromorphic continuation of the intertwining operator $J(\lambda )$. This result is not needed for the computation of $\eta_\mu (\lambda )$. Indeed, it is enough to know that $J(\lambda )$ is holomorphic on $\fa_\C^*(a_P)$. Then Theorem \ref{th:etaMu} implies the meromorphic continuation of $J(\lambda)=\CL$ at least at the level of $K$-finite functions, cf. \cite{KnappStein2}.
\end{remark}

\begin{remark} 
One can give a more general definition of the $\CosL$-transform. Note first that the definition of $|\Cos (b,c)|^\lambda $ in (\ref{defCos}) is valid for all $b\in \Gr_{p^\prime}(\K)$, $p^\prime \ge p$. For $p^\prime \le p$ one uses the symmetry relation 
$|\Cos (b,c)|^\lambda =|\Cos(b,c)|^\lambda $. For $p\le p^\prime$ define the Radon transform $R_{p^\prime ,p} :C^\infty (\Gr_{p^\prime})\to
C^\infty (\Gr_p)$ by
\[R_{p^\prime,p}(f)(\omega) :=\int_{\omega\subset \eta} f(\eta )\, d\sigma_{p^\prime,\omega}(\eta )\]
where $\sigma_{p^\prime}$ is a well defined probability measure on the set $\{\eta \mid \omega \subset \eta\}$. Then $R_{p^\prime ,p}$ is a $K$-intertwining operator whose $K$-spectrum is independent of $\lambda$. Furthermore, defining $\CL$ by (\ref{def-CosL}), one shows that, with the obvious notation, 
\[\CL_{p^\prime,p}=\CL \circ R_{p^\prime,p}\, .\]
The determination of the $K$-spectrum of the more general $\CosL$-transform is then reduced to the determination of the $\lambda$-independent $K$-spectrum of $R_{p^\prime,p}$ and the determination of the $\lambda$-dependent $K$-spectrum of $\CL$. We note that obviously $R_{p^\prime ,p}$ is zero on all $K$-types $\mu\in\Lambda^+(\cB_{p^\prime})\setminus \Lambda^+(\cB_p)\cap \Lambda^+(\cB_{p^\prime})$, or equivalently, have a nonzero entry $m_j$ with $p<j\le p^\prime$. As our purpose is to relate the $\CosL$-transform with the standard intertwining operators, we do not go further into this interesting topic, but refer to \cite{Zhang09} and the references therein for further discussion.
\end{remark}

\section{The $\mathbf{\SinL}$-transform}\label{se:SinL}
\noindent
In this section we introduce the $\SinL$-transform and determine its $K$-spectrum. We assume that $p=q$. This in particular implies that $n+1$ is even and $p=\frac{n+1}{2}$. If $\omega\in \Gr_p(\K)$ then 
$\omega^\perp$ denotes the orthogonal complement of $\omega$. Define the $\SinL$-transform $\SLa :C^\infty (\cB )\to C^\infty (\cB)$ by
\[\SLa (f)(\omega ):=\CL (f)(\omega^\perp)\, .\] 

This transform has been studied by B. Rubin for more general $p$, see \cite{Rubin2010,Rubin2011} and references therein, and by G. Zhang \cite{Zhang09}. In particular G. Zhang shows, using the $\oN$ realization of the principal series representations, that the $\SinL$-transform is a Knapp-Stein intertwining operator and uses this fact to determine the complementary series for $\pi_\lambda$. 

As $p=(n+1)/2$ we have $H_o=\begin{pmatrix} \frac{1}{2}\rI_p & 0\\ 0 & -\frac{1}{2}\rI_p\end{pmatrix}$.
Let
\[w= \begin{pmatrix} 0 & -\rI_p \\ \rI_p & 0\end{pmatrix}=\exp\left(Y\left(\frac{\pi}{2},\ldots ,\frac{\pi}{2}\right)\right)\, .\]
Then $wH_ow^{-1}=\Ad (w)H_0=-1$. It follows that
$\Ad (w)\oN = N$. 
\begin{lemma} For $f\in C^\infty (\cB)$ define $A (\lambda )f(k\cdot b_o):= J(\lambda )(f)(kw\cdot b_o)=R(w)J(\lambda )f(k\cdot b_o)$, where $R(w)$ denotes the right translation by $w$. Then $A(\lambda ): C^\infty (\cB)\to C^\infty (\cB)$. Then $\lambda \mapsto A(\lambda)$ is a meromorphic family of intertwining operators: $A (\lambda )\circ \pi_{\lambda }(x)= \pi_{-\lambda }(x)A (\lambda )$.
\end{lemma}
\begin{proof} We have $w^{-1}Lw=L$ as $\Ad (w)$ normalizes $\fa$. Thus 
$kmw\cdot b_o=kw\cdot b_o$ if $m\in L$ and $k\in K$. Thus $A(\lambda )$ is well defined. The intertwining property follows from Theorem \ref{thm:VW} and the fact that $w^{-1}\oN w=N$ and $\Ad (w)|_{\fa}=-1$.
\end{proof}
\begin{theorem}\label{thm:SinL} If $\lambda\in \fa_\C^* (a_P)$, then $A (\lambda )=\SLa$. 
\end{theorem}
\begin{proof} We have by definition
\begin{equation}\label{eq:Al1}
A (\lambda )f (k\cdot b_o)=\int_{K}|\Cos (h\cdot b_o, kw\cdot b_o)|^{\lambda -\rho} f(h)\, dh\, .
\end{equation}
The claim follows if we can show that $(k\cdot b_o)^{\perp} = kw\cdot b_o$. As $\SU (n+1,\K )$ acts by orthogonal transformations it is enough to show that $b_o^\perp = w\cdot b_o$. But thit follows from
\[w (x_1,\ldots ,x_p,0,\ldots , 0)^t=(0,\ldots , 0, x_1,\ldots ,x_p)^t\, .\qedhere \]
\end{proof}

As for the $\CosL$-transform, it follows that for each $\mu \in \LB$ there exists a meromorphic function $\nu_\mu (\lambda )$ such that
\[\SLa|_{L^2_\mu (\cB )}=\nu_\mu (\lambda) \Id_{L^2_\mu (\cB )}\, .\]

As $w^2\in L$ it follows that $R(w)$ defines a unitary $K$-intertwining operator of order two on each of the $K$-types $L^2_\mu (\cB)$. Thus $R(w)|_{L^2_\mu (\cB)}=r (\mu )\id_{L^2_\mu (\cB)}$ where $r (\mu )$ is a constant independent of $\lambda$, $r(\mu )^2=\pm 1$ and 
$\nu_\mu (\lambda )=r(\mu )\eta_\mu (\lambda )$.

\begin{lemma}\label{le:rmu} Let $\mu \in \LB$. Then $r(\mu )=(-1)^{|\mu |/2}$.
\end{lemma}
\begin{proof} Let $e_\mu\in V_\mu^L$ be as in Section \ref{section2}. Then $\pi_\mu (w)e_\mu =r (\mu )e_\mu$. Let $v_\mu$ be a highest weight vector such that $( v_\mu, e_\mu)_{V_\mu}=1$. Then
\[ r(\mu) =(v_\mu, \pi_\mu (w)e_\mu)
=(\pi_\mu (w^{-1})v_\mu,e_\mu)
= e^{-\frac{\pi i}{2}\sum m_j}(v_\mu, e_\mu)
= (-1)^{|\mu |/2}\, .\qedhere
\]
\end{proof}

\begin{theorem}\label{th:KspectrumSL} The $K$-spectrum of the $\SinL$-transform is given by 
\[\nu_\mu (\lambda )=\frac{\Gamma_{p,d}(\rho )}{\Gamma_{p,d}(\rho/2)}\, \frac{\Gamma_{p,d} \left(\frac{\lambda }{2}\right)\Gamma_{p,d}\left(\frac{1}{2}(-\lambda +\mu +\rho)\right)}{\Gamma_{p,d}\left(\frac{1}{2}(-\lambda +\rho)\right)\Gamma_{p,d}\left(\frac{1}{2}(\lambda +\mu +\rho)\right)}\]
with $\rho =dp$.
\end{theorem}
\begin{proof} This follows from Theorem \ref{th:etaMu} and Lemma \ref{le:rmu}. \end{proof}
\bigskip

\noindent 
\textbf{Acknowledgements:} The authors would like to thank B. Rubin for helpful discussions about the $\CosL$-transform, his work and the historical development of the subject. The first named author would also like to thank the harmonic analysis students at LSU to give him the opportunity to give an extended series of lectures on the subject in their student seminar. This was very helpful to clarify the ideas and the presentation.

\end{document}